\documentclass{birkjour}
\usepackage[utf8]{inputenc}
\usepackage{amsmath}
\usepackage{amsthm}
\usepackage{amsfonts}
\usepackage{amssymb}
\usepackage{hyperref}

\makeatletter
\DeclareRobustCommand*\cal{\@fontswitch\relax\mathcal}
\makeatother

 \newtheorem{thm}{Theorem}[section]
 
 \newtheorem{lem}[thm]{Lemma}
 \newtheorem{prop}[thm]{Proposition}
 \theoremstyle{definition}
 
 \theoremstyle{remark}
 \newtheorem{rem}[thm]{Remark}
 
 \newtheorem*{notation}{Notation}
 \numberwithin{equation}{section}

\begin{document}

%
%
%
%
%
%
%
%
%

\title[Nonlinear Schrödinger equation on the half-line]
 {Nonlinear Schrödinger equation on the half-line with nonlinear boundary condition}

\author{Ahmet Batal}

\address{Department of Mathematics\\ Izmir Institute of Technology\\Izmir, TURKEY}

\email{ahmetbatal@iyte.edu.tr}

\thanks{The work of Türker Özsarı was supported by the Izmir Institute of Technology
under the BAP grant 2015İYTE43.}
\author{Türker Özsarı}
\address{Department of Mathematics\\ Izmir Institute of Technology\\Izmir, TURKEY}
\email{turkerozsari@iyte.edu.tr}
\subjclass{35Q55; 35A01; 35A02; 35B30}

\keywords{}

\date{}
\dedicatory{}

\begin{abstract}
In this paper, we study the initial boundary value problem for nonlinear Schrödinger equations on the half-line with nonlinear boundary conditions of type $u_x(0,t)+\lambda|u(0,t)|^ru(0,t)=0,$ $\lambda\in\mathbb{R}-\{0\}$, $r> 0$.  We discuss the local well-posedness when the initial data $u_0=u(x,0)$ belongs to an $L^2$-based inhomogeneous Sobolev space $H^s(\mathbb{R}_+)$  with $s\in \left(\frac{1}{2},\frac{7}{2}\right)-\{\frac{3}{2}\}$.  We deal with the nonlinear boundary condition by first studying the linear Schrödinger equation with a time-dependent inhomogeneous Neumann boundary condition $u_x(0,t)=h(t)$ where $h\in H^{\frac{2s-1}{4}}(0,T)$.  This latter problem is studied by adapting the method of Bona-Sun-Zhang \cite{BonaSunZhang2015} to the case of inhomogeneous Neumann boundary conditions.
\end{abstract}

\maketitle

\section{Introduction and Main Result}\label{Intro}
The nonlinear Schrödinger equation (NLS) is a fundamental dispersive partial differential equation. NLS can be used in many physical nonlinear systems such as quantum many body systems, optics, hydrodynamics, acoustics, quantum condensates, and heat pulses in solids.

In this article, we consider the following nonlinear Schrödinger equation with nonlinear boundary condition on the (right) half-line.
\begin{equation}\label{MainProb} \left\{ \begin{array}{ll}
         i\partial_t u + \partial_x^2u+k|u|^pu= 0, & \mbox{$x\in \mathbb{R}_+$, $t\in (0,T)$},\\
         u(x,0)=u_0(x),\\
         \partial_xu(0,t)+\lambda|u(0,t)|^ru(0,t)=0,\end{array} \right.
         \end{equation} where $u(x,t)$ is a complex valued function, the real variables
         $x$ and $t$ are space and time coordinates, and $\partial_t,\partial_x$ denote partial derivatives with respect to time and space.  The constant parameters satisfy $k,\lambda\in\mathbb{R}-\{0\}$, and $p,r> 0$.  When $\lambda=0$, the boundary condition reduces to the classical homogeneous Neumann boundary condition. When $r=0$, the boundary condition is the classical homogeneous Robin boundary condition.  When $\lambda$ and $r$ are both non-zero as in the present case, the boundary condition can be considered as a nonlinear variation of the Robin boundary condition.

Our main goal is to solve the classical local well-posedness problem for \eqref{MainProb}.  More precisely, we will prove the local existence and uniqueness for \eqref{MainProb} together with the continuous dependence of solutions on the initial data $u_0$, which is taken from an $L^2-$based inhomogeneous Sobolev space $H^s(\mathbb{R}_+)$ with ${s\in\left(\frac{1}{2},\frac{7}{2}\right)-\left\{\frac{3}{2}\right\}}.$  We will also deduce a blow-up alternative for the solutions of \eqref{MainProb} in the $H^s$-sense.

The well-posedness problem will be considered in the function space $X_T^s$, which is the set of those elements in $$C([0,T];H^s(\mathbb{R}_+))\cap C(\mathbb{R}_+^x;H^{\frac{2s+1}{4}}(0,T))$$ that are bounded with respect to the norm ${\|\cdot\|_{X_T^s}}$.   This norm is defined by \begin{equation}\label{Xtsnorm}\|f\|_{X_T^s}:=\sup_{t\in[0,T]}\|f(\cdot,t)\|_{H^s(\mathbb{R_+})}+\sup_{x\in\mathbb{R}_+}\|f(x,\cdot)\|_{H^{\frac{2s+1}{4}}(0,T)}.\end{equation}

It is well-known that the trace operators $\gamma_0:u_0\rightarrow u_0(0)$ and $\gamma_1:u_0\rightarrow u_0'(0)$ are well-defined on $H^s(\mathbb{R}_+)$ when $s>\frac{1}{2}$ and $s>\frac{3}{2}$, respectively.  Therefore, both $u_0(0)$ and $u_0'(0)$ make sense if $s>\frac{3}{2}$.  Hence, we will assume the compatibility condition $u_0'(0)=-\lambda|u_0(0)|^ru_0(0)$ when $s>\frac{3}{2}$ on the initial data to comply with the desire that the solution be continuous at $(x,t)=(0,0)$.

Now we can state our main result.
\begin{thm}[Local well-posedness]\label{MainResult1} Let $T>0$ be arbitrary, $s\in \left(\frac{1}{2},\frac{7}{2}\right) -\left\{\frac{3}{2}\right\}$,  $p,r>0$, $k,\lambda\in\mathbb{R}-\{0\}$, $u_0\in H^s(\mathbb{R_+})$ together with ${u_0'(0)=-\lambda|u_0(0)|^ru_0(0)}$ whenever $s>\frac{3}{2}$.  We in addition assume the following restrictions on $p$ and $r$:
\begin{itemize}
  \item[(A1)]  If $s$ is integer, then $p\ge s$ if $p$ is an odd integer and $[p]\ge s-1$ if $p$ is non-integer.
  \item[(A2)]  If $s$ is non-integer, then $p>s$ if $p$ is an odd integer and $[p]\ge [s]$ if $p$ is non-integer.
  \item[(A3)] $r>\frac{2s-1}{4}$ if $r$ is an odd integer and $[r]\ge \left[\frac{2s-1}{4}\right]$ if $r$ is non-integer.
\end{itemize}  Then, the following hold true.
\begin{itemize}
  \item[(i)] Local Existence and Uniqueness: There exists a unique local solution $u\in X_{T_0}^s$ of \eqref{MainProb} for some $T_0=T_0\left(\|u_0\|_{H^s(\mathbb{R}_+)}\right)\in (0,T]$.
  \item[(ii)] Continuous Dependence: If $B$ is a bounded subset of $H^s(\mathbb{R}_+)$, then there is $T_0>0$ (depends on the diameter of $B$) such that the flow $u_0\rightarrow u$ is Lipschitz continuous from $B$ into $X_{T_0}^s$.
  \item[(iii)] Blow-up Alternative: If $S$ is the set of all $T_0\in (0,T]$ such that there exists a unique local solution in $X_{T_0}^s$, then whenever $\displaystyle T_{max}:=\sup_{T_0\in S}T_0<T$, it must be true that ${\displaystyle\lim_{t\uparrow T_{max}}\|u(t)\|_{H^s(\mathbb{R}_+)}=\infty}$.
\end{itemize}
 \end{thm}

\begin{rem} If $s=1$ or $p$ is even, then the assumptions on $p$  given in (A1)-(A2) in Theorem \ref{MainResult1} are redundant.  The same remark applies to $r$ when $s=5/2-\epsilon$ or $r$ is even.
\end{rem}

\begin{rem}In the above theorem, when $s\ge 2$, the equation is understood in the $L^2-$sense.  However, if $s<2$, the equation should be understood in the distributional sense, namely in the sense of $H^{s-2}(\mathbb{R}_+)$.  For low values of $s$, the boundary and the initial condition can be understood in the sense of Definition 2.2 in \cite{BonaSunZhang2015}.\end{rem}

\subsection*{Literature Overview}

To the best of our knowledge, the model given in \eqref{MainProb} has only been studied in the case that $k=0, \lambda=1$, and $r>0$ by Ackleh-Deng \cite{AcklehDeng2004}.  In \cite{AcklehDeng2004}, the main equation was only linear.  More precisely, the authors studied the following.
  \begin{equation}\label{ADModel} \left\{ \begin{array}{ll}
         i\partial_t u + \partial_x^2u= 0, & \mbox{$x\in \mathbb{R}_+$, $t\in (0,T)$},\\
         u(x,0)=u_0(x),\\
         \partial_xu(0,t)+|u(0,t)|^ru(0,t)=0.\end{array} \right.
         \end{equation}
Ackleh-Deng \cite{AcklehDeng2004} proved that if $u_0\in H^3(\mathbb{R}_+)$, then there is $T_0>0$ such that \eqref{ADModel} possesses a unique local solution $u\in C([0,T_0);H^1(\mathbb{R}_+))$.  Moreover, it was shown in \cite{AcklehDeng2004} that (large) solutions with negative initial energy blow-up if $r\ge 2$ and are global otherwise.  Therefore, $r=2$ was considered the critical exponent for \eqref{ADModel}.  Obtaining local existence and uniqueness consisted of two steps.  First, the authors studied the linear Schrödinger equation with an inhomogeneous Neumann boundary condition on the half-line.  Secondly, they used a contraction argument once the representation formula for solutions was restricted to the boundary point $x=0$.  In other words, the contraction argument was used on a function space which included only time dependent elements.    Unfortunately, the same technique cannot be applied in the presence of the nonlinear source term $f(u)=k|u|^pu$ in the main equation.  The reason is that even if the representation formula can still be restricted to the point $x=0$, the sought after fixed point in the representation formula would also depend on the space variable.  Therefore, one can no longer use a simple contraction argument on a function space which includes only time dependent elements.  We are thus motivated to use a contraction argument on a function space which includes elements that depend on both time and space variables.  Of course, this requires nice linear and nonlinear space-time estimates.

The drawback of the technique used in \cite{AcklehDeng2004} is that the initial data has been assumed to be too smooth compared to the regularity of the solutions obtained.  It is well-known from the theory of the linear Schrödinger equation that solutions are of the same class as the initial state.  From this point of view, the generation of $H^1$ solutions with $H^3$ data seems far from optimal.  We are thus inclined to obtain a regularity theory which shows that $H^s$ initial data generates $H^s$ solutions.

Regarding nonlinear boundary conditions, we are aware of very few other results for Schrödinger equations, see for example \cite{LasieckaTriggiani2006} and \cite{Wellington}.  In \cite{LasieckaTriggiani2006}, the authors study the Schrödinger equation with nonlinear, attractive, and dissipative boundary conditions of type $\frac{\partial u}{\partial \nu}=ig(u)$ where $g$ is a monotone function with the property that the corresponding evolution operator generates a strongly continuous contraction semigroup on the $L^2$-level.  The more recent paper \cite{Wellington} studies Schrödinger equation with Wentzell boundary conditions.  This work also uses the fact that the Wentzell boundary condition provide a semigroup in an appropriate topology.  In the present case, due to the fact that $\lambda$ is not a purely imaginary number, the problem does not have a monotone structure, and the method of \cite{LasieckaTriggiani2006}, \cite{Wellington} cannot be applied here.

A common strategy for proving well-posedness of solutions to PDEs with nonlinear terms relies on two classical steps: (1) obtain a good linear theory with non-homogeneous terms; (2) establish local well-posedness for the nonlinear model by a fixed point argument.

Obtaining a good linear theory with non-homogeneous terms is a subtle point for boundary value problems, especially those with low-regularity boundary data.  One might attempt to extend the boundary data into the domain and homogenize the boundary condition.  However, this approach in general requires a high regularity boundary data (\cite{Bu1994}, \cite{Bu2000}), as opposed to the rough boundary situation as in the present paper for low values of $s$.  There are different approaches one can follow to study a linear PDE with an inhomogeneous boundary data on the half-line without employing an extension-homogenization approach, though.  For example, Colliander-Kenig \cite{CollianderKenig2002} used a technique on the KdV equation by replacing the given initial-boundary value problem with a forced initial value problem where the forcing is chosen in such a way that the boundary condition is satisfied by inverting a Riemann-Liouville fractional integral.  Holmer \cite{Holmer2005} applied this technique on nonlinear Schrödinger equations with inhomogeneous Dirichlet boundary conditions on the half line.  A second approach is to obtain norm estimates on solutions by using a representation formula, which can be easily  obtained through a Laplace/Fourier transform.  This technique has been used for example by Kaikina in \cite{Kaikina2013} for nonlinear Schrödinger equations with inhomogeneous Neumann boundary conditions and by Bona-Sun-Zhang in \cite{BonaSunZhang2015} for inhomogeneous Dirichlet boundary conditions.  In \cite{Kaikina2013}, the well-posedness result assumes the smallness of the given initial-boundary data while the results of \cite{BonaSunZhang2015} have global character in this sense.

Although nonlinear Schrödinger equations with inhomogeneous boundary conditions have been studied to some extent, most of these papers were devoted to inhomogeneous Dirichlet boundary conditions; see \cite{CarrollBu1991}, \cite{Bu1992}, \cite{Bu1994}, \cite{BuShullWangChu2001}, \cite{Bu2000A}, \cite{StraussBu2001}, \cite{BuTsutayaZhang2005}, \cite{Holmer2005}, \cite{Ozsari2011}, \cite{Kaikina2013D}, \cite{BonaSunZhang2015}, \cite{Ozsari2012}, \cite{Ozsari2015}.  There are relatively less results on inhomogeneous Neumann boundary conditions; see \cite{Bu1994}, \cite{Bu2000}, \cite{Kaikina2013}, \cite{Ozsari2013}, \cite{Ozsari2015}.  In \cite{Bu1994} and \cite{Bu2000}, well-posedness is obtained under smooth boundary data.  Relatively less smooth boundary data was treated in \cite{Ozsari2015} using Strichartz estimates, but the regularity results were not optimal.    In \cite{Kaikina2013}, the smallness of initial and boundary data was crucial.  In \cite{Ozsari2013}, the focus was on the existence of weak solutions, and questions concerning continuity in time, uniqueness, and continuous dependence on data were not studied.  In the present paper, we draw a more complete and optimal well-posedness picture where the spatial domain is half-line.

\subsection*{Orientation}

In this paper, we will follow a step-by-step approach to prove Theorem \ref{MainProb}:

\emph{Step 1}: We will first study the linear Schrödinger equation with  inhomogeneous terms both in the main equation and in the boundary condition. This problem is written in \eqref{MainLinNonHom}.  Our aim in this step is to derive optimal norm estimates with respect to regularities of the initial state $u_0$, boundary data $h$, and nonhomogeneous source term $f$.  This linear theory is constructed in Section \ref{SectionLin} by adapting the method of \cite{BonaSunZhang2015} to nonhomogeneous Neumann boundary conditions.

\emph{Step 2}: In the second step, we will replace the nonhomogeneous source term $f=f(x,t)$ in \eqref{MainLinNonHom} with $f=f(u)=k|u|^pu$ as in \eqref{MainProb2}.  We will use a contraction mapping argument to prove the existence and uniqueness of local solutions  together with continuous dependence on data.  The blow-up alternative will be obtained via a classical extension-contradiction argument. This step is treated in Sections \ref{LocExt} - \ref{BlowSec}.

\emph{Step 3}: In this step, we will replace the boundary data $h=h(x,t)$ in \eqref{MainLinNonHom} with $h=h(u)=-\lambda|u(0,t)|^ru(0,t)$, and $f$ with $k|u|^pu$.  Arguments similar to those in Step 2 will eventually give the well-posedness in the presence of nonlinear boundary conditions.  The only difference is that the contraction argument must be adapted to deal with the nonlinear effects due to the nonlinear boundary source.  This  is given in Section \ref{SectionNonlin2}.

\begin{rem}Step 2 is indeed optional.  One can directly run the contraction and blow-up arguments with nonlinear boundary conditions.  However, it is useful to include the general theory of nonlinear Schrödinger equations with inhomogeneous Neumann boundary conditions to study other related problems in the future. \end{rem}

\section{Linear nonhomogeneous model}\label{SectionLin}
In this section, we study the nonhomogeneous linear Schrödinger equation with nonhomogeneous Neumann boundary condition. We will later apply this linear theory to obtain the local well-posedness for nonlinear Schrödinger equations first with inhomogeneous Neumann boundary conditions and then with nonlinear boundary conditions.  In order to obtain a sufficiently nice linear theory, we adapt the method presented for nonhomogeneous Dirichlet boundary conditions in \cite{BonaSunZhang2015} to the case with nonhomogeneous Neumann boundary conditions.


We consider the following linear model
\begin{equation}\label{MainLinNonHom} \left\{ \begin{array}{ll}
         i\partial_t u + \partial_x^2u + f = 0, x\in \mathbb{R}_+, t\in (0,T),\\
         u(x,0)=u_0, \partial_xu(0,t)= h(t),\end{array} \right.
\end{equation} where $f$ and $h$ lie in appropriate  function spaces.

\subsection{Compatibility conditions} Suppose $u_0\in H^{s}(\mathbb{R}_+)$, $h\in H^{\frac{2s-1}{4}}(0,T)$ in \eqref{MainLinNonHom}. It is well-known from the trace theory that both $u_0'(0)$ and $h(0)$ make sense when $s>\frac{3}{2}$.  Therefore, one needs to assume the zeroth order compatibility condition $$u_0'(0)=h(0)$$ when $s\in\left(\frac{3}{2},\frac{7}{2}\right)$ in order to get continuous solutions at $(x,t)=(0,0)$.  As the value of $s$ gets higher, one needs to consider more compatibility conditions.  For example, if $s\in \left(2k+\frac{3}{2},2(k+1)+\frac{3}{2}\right)$ ($k\ge 1$), then the $k$-th order compatibility condition is defined inductively: $$\varphi_0=u_0, \varphi_{n+1}=i(\partial_t^{n}f|_{t=0}+\partial_x^2\varphi_n),$$ $$\partial_t^kh|_{t=0}=\partial_x\varphi_k|_{x=0}$$ provided that $f$ is also smooth enough for traces to make sense.  If one wants to add the end point cases $s=2k+\frac{3}{2}$ to the analysis, then \emph{global compatibility conditions} must be assumed (see for example \cite{Audiard2015} for a discussion of local and global compatibility conditions in the case of Dirichlet boundary conditions).

\subsection{Boundary operator}
We will first deduce a representation formula for solutions of the following linear model with an inhomogeneity on the boundary.

\begin{equation}\label{LinNonHom} \left\{ \begin{array}{ll}
         i\partial_t u + \partial_x^2u = 0, x\in \mathbb{R}_+, t\in (0,T),\\
         u(x,0)=0, \partial_xu(0,t)= h(t).\end{array} \right.
\end{equation}

We will study the above model by constructing an evolution operator which acts on the boundary data.  We will start by taking a Laplace (in time) - Fourier (in space) transform of the given model.   In order to do that, we will first extend the boundary data to the whole line utilizing the following lemma.
\begin{lem}[Extension]\label{Extension}Let $s\in \left(\frac{1}{2},\frac{7}{2}\right)-\{\frac{3}{2}\}$, $h\in H^\frac{2s-1}{4}(0,T)$ with $h(0)=0$ if $s>\frac{3}{2}$.  Then, there exists $h_e\in H^\frac{2s-1}{4}$ with compact support in $[0,2T+1)$ which extends $h$ so that ${\cal{H}}(t):=\int_{-\infty}^th_e(s)ds$ also has compact support in $[0,2T+1)$ and $\|{\cal{H}}\|_{H^{\frac{2s+3}{4}}}\le C(1+T)\|h\|_{H^\frac{2s-1}{4}(0,T)}$ for some $C>0$ which is independent of $T$. \end{lem}

\begin{proof}If $\frac{1}{2}<s<\frac{3}{2}$, we have $0<\frac{2s-1}{4}<\frac{1}{2}$. Now we take the zero extension of $h$ onto $\mathbb{R}$, say we get $h_0$.  Then we set $h_e(t):=h_0(t)-h_0(t-T).$

If $s\in \left(\frac{3}{2},\frac{7}{2}\right)$, then $\frac{1}{2}<\frac{2s-1}{4}<\frac{3}{2}$. In this case, we first take an extension $h_{A}$ of $h$ onto $\mathbb{R}$ so that $\|h_{A}\|_{H^\frac{2s-1}{4}}\le 2\|h\|_{H^\frac{2s-1}{4}(0,T)}$ by using the fact that \begin{equation}\label{infDef}\|h\|_{H^\frac{2s-1}{4}(0,T)}:=\inf\left\{\|\phi\|_{H^\frac{2s-1}{4}}: \phi\in H^\frac{2s-1}{4}, \phi|_{(0,T)}=h\right\}.\end{equation}  Secondly, the restriction $h_B:=h_A|_{(0,\infty)}\in H^\frac{2s-1}{4}(0,\infty)$ will satisfy {$\|h_{B}\|_{H^\frac{2s-1}{4}(0,\infty)}\le \|h_A\|_{H^\frac{2s-1}{4}}$}.  Now we can take the zero extension, say $h_C$, of $h_B$ onto $\mathbb{R}$ so that $\|h_{C}\|_{H^\frac{2s-1}{4}}\le C\|h_B\|_{H^\frac{2s-1}{4}(\mathbb{R}_+)}$ with $C$ independent of $T$.  By the previous inequalities, we get $\|h_C\|_{H^\frac{2s-1}{4}}\le C\|h\|_{H^\frac{2s-1}{4}(0,T)}$ with $C$ independent of $T$.  Then we pick a function $\eta\in C_c^\infty(\mathbb{R})$ so that $\eta=1$ on $(0,T)$ and $\eta=0$ on $[T+1/2,\infty)$. Now we consider $h_1=\eta h_C$, which is of course in $H^\frac{2s-1}{4}$, since $H^\frac{2s-1}{4}$ is a Banach algebra when $s>\frac{3}{2}$. Finally, we set $h_e(t)= h_1(t)-h_1(t-T-1/2).$

Note that $\|h_e\|_{H^\frac{2s-1}{4}}\leq C \|h\|_{H^\frac{2s-1}{4}(0,T)}$ where the positive constant $C$ does not depend on $T$, since all the extensions in the above paragraph and the multiplication by $\eta$ are continuous operators between corresponding Sobolev spaces whose norms do not depend on the initial domain $(0,T)$. Moreover, we set up $h_e$ in such a way that its average is zero. Hence, its antiderivative ${\cal{H}}(t):=\int_{-\infty}^th_e(s)ds$ is compactly supported and therefore belongs to the space $H^{\frac{2s+3}{4}}.$

Since $\cal{H}$ is compactly supported with support in $[0, 2T+1)$ by the Poincaré inequality we have $\|{\cal{H}}\|_{L^2}\leq (2T+1)\|h_e\|_{L^2}.$ Hence \begin{multline}\|{\cal{H}}\|_{H^{\frac{2s+3}{4}}}\simeq \|D^{\frac{2s+3}{4}}{\cal{H}}\|_{L^2}+ \|{\cal{H}}\|_{L^2}\leq C\|D^{\frac{2s-1}{4}}{h_e}\|_{L^2}+(2T+1)\|h_e\|_{L^2}\\\leq C(1+T)\|h_e\|_{H^\frac{2s-1}{4}}\leq C(1+T)\|h\|_{H^\frac{2s-1}{4}(0,T)}\end{multline} for some $C>0$.
\end{proof}

Now we consider the following model, which is an extended-in-time version of \eqref{LinNonHom}.
\begin{equation}\label{ELinNonHom} \left\{ \begin{array}{ll}
         i\partial_t u_{e} + \partial_x^2u_{e} = 0, x\in \mathbb{R}_+, t>0,\\
         u_{e}(x,0)=0, \partial_xu_{e}(0,t)= h_{e}(t)\end{array} \right.
\end{equation} where $h_e$ is the extension of $h$, as in Lemma \ref{Extension}.

We first take the Laplace transform of \eqref{ELinNonHom} in $t$ to get
\begin{equation}\label{LT-LinNonHom} \left\{ \begin{array}{ll}
         i\lambda\tilde{u}_{e}(x,\lambda) + \partial_x^2\tilde{u}_{e}(x,\lambda)= 0,\\
         \tilde{u}_{e}(+\infty,\lambda)=0, \partial_x\tilde{u}_{e}(0,\lambda)= \tilde{h}_e(\lambda)\end{array} \right.
\end{equation} with $\text{Re } \lambda>0$, where $\tilde{u}_e$ denotes the Laplace transform of ${u}_e$. The solution of \eqref{LT-LinNonHom} is $$\tilde{u}_{e}(x,\lambda)=\frac{1}{r(\lambda)} \exp(r(\lambda)x)\tilde{h}_{e}(\lambda)$$ where $\text{Re }r(\lambda)$ solves $i\lambda+r^2=0$ together with $\text{Re }r<0$.  Then, $$u_{e}(x,t)=\frac{1}{2\pi i}\int_{-\infty i+\gamma}^{+\infty i+\gamma}\exp(\lambda t)\frac{1}{r(\lambda)} \exp(r(\lambda)x)\tilde{h}_{e}(\lambda)d\lambda,$$ where $\gamma>0$ (fixed), solves \eqref{LT-LinNonHom}.  By passing to the limit in $\gamma$ as $\gamma\rightarrow 0$ and applying change of variables, we can rewrite $u(x,t)$ as follows:

\begin{multline}u_{e}(x,t)=\frac{1}{i\pi}\int_0^\infty\exp(-i\beta^2 t+i\beta x)\tilde{h}_{e}(-i\beta^2)d\beta\\
-\frac{1}{\pi}\int_0^\infty\exp(i\beta^2 t-\beta x)\tilde{h}_{e}(i\beta^2)d\beta.\end{multline}

Note, that $u:=u_e|_{[0,T)}$ is a solution of \eqref{LinNonHom}.  We define $\nu_1(\beta):=\frac{1}{i\pi}\tilde{h}_{e}(-i\beta^2)$ for $\beta\ge 0$ and zero otherwise.  Let $\phi_{h_e}$ be the inverse Fourier transform of $\nu_1$, that is $\hat{\phi}_{h_{e}}(\beta)=\nu_1(\beta)$ for $\beta\in \mathbb{R}$.  Similarly, we define $\nu_2(\beta):=-\frac{1}{\pi}\tilde{h}_{e}(i\beta^2)$ for $\beta\ge 0$ and zero otherwise.  Let $\psi_{h_e}$ be the inverse Fourier transform of $\nu_2$, that is $\hat{\psi}_{h_{e}}(\beta)=\nu_2(\beta)$ for $\beta\in \mathbb{R}$.  Now, for $x\in\mathbb{R}_+$, we can write
$$u_{e}(x,t)=[W_b(t)h_{e}](x):=[W_{b,1}(t)h_{e}](x)+[W_{b,2}(t)h_{e}](x)$$ where
$$[W_{b,1}(t)h_{e}](x):=\int_{-\infty}^\infty\exp(-i\beta^2 t+i\beta x)\hat{\phi}_{h_{e}}(\beta)d\beta$$
and
$$[W_{b,2}(t)h_{e}](x):=\int_{-\infty}^\infty\exp(i\beta^2 t-\beta x)\hat{\psi}_{h_{e}}(\beta)d\beta.$$

Note that we can extend $W_{b,1}(t)h_{e}$ to $\mathbb{R}$ without changing its definition.  For such an extension we have the following lemma:
\begin{lem}\label{Wb1Lem} $u(x,t)=[W_{b,1}(t)h_{e}](x)$ solves the initial value problem $$i\partial_tu+\partial_x^2u=0, u(x,0)=\phi_{h_{e}}(x), x\in\mathbb{R},t\in\mathbb{R}_+.$$
\end{lem}
\begin{proof} By direct calculation, we have $$i\partial_tu+\partial_x^2u=[i(-i\beta^2)+(i\beta)^2][W_{b,1}(t)h_{e}](x)=0,$$ and $$u(x,0)={\mathcal{F}}^{-1}(\hat{\phi}_{h_{e}})(x)=\phi_{h_{e}}(x).$$
\end{proof}

We deduce from the above lemma that we can get space time estimates on $W_{b,1}(t)h_{e}$ by using the well-known linear theory of Schrödinger equations on $\mathbb{R}.$  These estimates are given in Section \ref{Rtheory}.  We extend $[W_{b,2}(t)h_{e}](x)$ to $\mathbb{R}$ by setting $$[W_{b,2}(t)h_{e}](x):=\int_{-\infty}^\infty\exp(i\beta^2 t-\beta |x|)\hat{\psi}_{h_{e}}(\beta)d\beta.$$  However, if $s>\frac{3}{2}$, then this extension would not be differentiable at $x=0$.  Therefore, if $s>\frac{3}{2}$, we cannot directly use the linear theory of Schrödinger equations on $\mathbb{R}$ to estimate various norms of the term $W_{b,2}(t)h_{e}$.  This makes it necessary to obtain space-time estimates for $W_{b,2}(t)h_{e}$ directly by using its definition.

The relation between regularities of $\phi_{h_{e}},\psi_{h_{e}}$ and the regularity of the boundary data $h$ is given by the following lemma.

\begin{lem}\label{phivsh} Let $s\ge \frac{1}{2}$, $h\in H^{\frac{2s-1}{4}}(0,T)$ such that $h(0)=0$ if $s>\frac{3}{2}$. Then, $\phi_{h_{e}}, \psi_{h_{e}}\in H^s$.
\end{lem}
\begin{proof} \begin{multline}\label{Phie1}\|\phi_{h_{e}}\|_{H^s}^2=\int_{-\infty}^\infty(1+\beta^2)^{s}|\hat{\phi}_{h_{e}}(\beta)|^2d\beta\\
=\frac{1}{\pi^2}\int_0^\infty(1+\beta^2)^{s}|\tilde{h}_{e}(-i\beta^2)|^2d\beta.\end{multline}

Upon change of variables, the last term in \eqref{Phie1} can be rewritten and estimated as follows.
\begin{multline}\frac{1}{2\pi^2}\int_0^\infty\frac{(1+\beta)^{s}}{\beta^{\frac{1}{2}}}|\tilde{h}_{e}(-i\beta)|^2d\beta\lesssim \frac{1}{\pi^2}\int_0^\infty{(1+\beta^2)^{\frac{2s+3}{4}}}|\hat{\mathcal{H}}(\beta)|^2d\beta\\
\lesssim \|\mathcal{H}\|_{H^{\frac{2s+3}{4}}}^2\end{multline} where we use the relationships, $\tilde{h}_{e}(-i\beta)=\hat{h}_e(\beta)$ and $\hat{h}_e(\beta)=i\beta\hat{\mathcal{H}}(\beta)$ in the first inequality.  The last estimate combined with Lemma \ref{Extension} implies that $\phi_{h_{e}}\in H^s$.  We can repeat the same argument for $\psi_{h_{e}}$, too.

\end{proof}

\begin{notation} A given pair $(q,r)$ is said to be \emph{admissible} if $\frac{1}{q}+\frac{1}{2r}=\frac{1}{4}$ for $q,r\ge 2$.
\end{notation}

Now, we will present several space-time estimates for the second part of the evolution operator $W_b(t)$.

\begin{lem}[Space Traces] Let $s\ge \frac{1}{2}$ and $T>0$.  Then, there exists $C>0$ (independent of $T$) such that \begin{equation}\label{Wb2Est2} \sup_{t\in [0,T]}\|W_{b,2}(\cdot)h_{e}\|_{H^s}\le C(1+T) \|h\|_{H^{\frac{2s-1}{4}}(0,T)}\end{equation}
for any $h\in H^{\frac{2s-1}{4}}(0,T)$ with $h(0)=0$ if $s>\frac{3}{2}$.
\end{lem}

\begin{proof} We can rewrite $[W_{b,2}(t)h](x)$ as $$[W_{b,2}(t)h_e](x):=\int_{-\infty}^\infty K_t(x,y){\psi}_{h_e}(y)dy=:\mathcal{K}(t)\psi_{h_e}$$ where $K_t(x,y)=\int_0^\infty\exp(i\beta^2 t-\beta |x|-iy\beta)d\beta.$  It is proven in \cite[Proposition 3.8]{BonaSunZhang2015} that  $$\|\mathcal{K}(t)\psi_{h_e}\|_{L^q(0,T;L^r)}\lesssim \|\psi_{h_e}\|_{L^2}$$ for an admissible $(q,r)$.  Similarly, taking one derivative in $x$ variable, one gets
$$\|\partial_x[\mathcal{K}(t)\psi_{h_e}]\|_{L^q(0,T;L^r)}\lesssim \|\partial_x[\psi_{h_e}]\|_{L^2}.$$  Now, one can interpolate and use the proof of Lemma \ref{phivsh} to obtain \begin{equation}\label{Wb2Est1}\|W_{b,2}(\cdot)h_{e}\|_{L^q(0,T;W^{s,r})}\lesssim \|\mathcal{H}\|_{H^{\frac{2s+3}{4}}}\end{equation} for $s\in \left[\frac{1}{2},1\right]$. For larger $s$, one can differentiate and interpolate again.  Finally, \eqref{Wb2Est2} follows by taking $r=2,q=\infty$ in \eqref{Wb2Est1}.  Now, \eqref{Wb2Est2} follows from \eqref{Wb2Est1} and Lemma \ref{Extension}.
\end{proof}

\begin{lem}[Time traces] Let $s\ge \frac{1}{2}$ and $T>0$.  Then, there exists $C>0$ (independent of $T$) such that \begin{equation}\label{Wb2EstTime1} \sup_{x\in\mathbb{R}_+}\|W_{b,2}(\cdot)h_{e}\|_{H^{\frac{2s+1}{4}}(0,T)}\le C(1+T)\|h\|_{H^{\frac{2s-1}{4}}(0,T)}\end{equation}
for any $h\in H^{\frac{2s-1}{4}}(0,T)$ with $h(0)=0$ if $s>\frac{3}{2}$.
\end{lem}
\begin{proof} This result is an application of Theorem 4.1 in \cite{KenigPonceVega1991}.  For $k\ge 0$ (integer)
\begin{multline}\label{dtEst}\|\partial_t^kW_{b,2}(\cdot){h_e}\|_{L^2_t}^2
=\int_{\mathbb{R}_+}\beta^{4k}\frac{|\hat{\psi}_{h_e}(\beta)|^2}{2\beta}d\beta\\\lesssim \int_\mathbb{R_+}(1+\beta^2)^{k+\frac{1}{2}}|\hat{\mathcal{H}}(\beta)|^2d\beta\lesssim\|\mathcal{H}\|_{H^{{k+\frac{1}{2}}}}. \end{multline}
Upon interpolation, the result follows in the case that $h$, $h_e$, and $\mathcal{H}$ are smooth, then a density argument finishes the proof.  Now, \eqref{Wb2EstTime1} follows from \eqref{dtEst} and Lemma \ref{Extension}.
\end{proof}

\subsection{Representation Formula}
We take an extension of $u_0$ to $\mathbb{R}$, say $u^*_0\in H^s$ such that $\|u_0^*\|_{H^s}\lesssim \|u_0\|_{H^s(\mathbb{R}^+)}$.  Therefore, $u=W_\mathbb{R}(t)u^*_0$ solves the problem
$$i\partial_tu + \partial_x^2u=0,u(0,t)=u^*_0(x), x,t\in \mathbb{R}$$  where $W_\mathbb{R}(t)$ is the evolution operator for the linear Schrödinger equation.  Similarly, if $f^*$ is an extension of $f$, then the solution of the non-homogeneous Cauchy problem
$$iu_t+u_{xx}=f^*(x,t), u(x,0)=0, x,t\in \mathbb{R}$$ can be written as $$u(x,t)=-i\int_0^tW_\mathbb{R}(t-\tau)f^*(\tau)d\tau.$$

Therefore, if we define \begin{equation}\label{RepForm}u_e(x,t)=W_\mathbb{R}(t)u^*_0-i\int_0^tW_\mathbb{R}(t-\tau)f^*(\tau)d\tau+W_b([h-g-p]_e(t))\end{equation} with $$g(t)=\partial_x W_\mathbb{R}(t)u^*_0|_{x=0}$$ and $$p(t)=-i\partial_x\int_0^tW_\mathbb{R}(t-\tau)f^*(\tau)d\tau|_{x=0},$$ then $u={u_e}|_{[0,T)}$ will solve
\begin{equation}\label{LinNonHom2} \left\{ \begin{array}{ll}
         i\partial_t u + \partial_x^2u= f, t\in (0,T)$, $x\in \mathbb{R}_+,\\
         u(x,0)=u_0, \partial_xu(0,t)= h(t).\end{array} \right.
\end{equation}

In the formula we have given, $g(t)$ and $p(t)$ make sense only if $s>3/2$.  In other cases, we take those boundary traces equal to zero in the representation formula \eqref{RepForm}.

\subsection{Space-time estimates on $\mathbb{R}$}\label{Rtheory}

We will utilize the following space and time estimates on $\mathbb{R}$ for the evolution operator of the linear Schrödinger equation \cite{Cazenave2003}.  Note that these estimates can be directly applied to the first part $W_{b,1}$ of the boundary evolution operator.

\begin{lem}\label{RTheory1} Let $s\in \mathbb{R}$ , $T>0$, $\phi\in H^s$, and $u:=W_{\mathbb{R}}\phi$.  Then, there exists $C=C(s)$ such that
\begin{equation}\label{Wb2Est1onR}\sup_{t\in[0,T]}\|u(\cdot,t)\|_{H^s}+\sup_{x\in\mathbb{R}}\|u(x,\cdot )\|_{H^{\frac{2s+1}{4}}(0,T)}\le C \|\phi\|_{H^s}.\end{equation}
\end{lem}

\begin{lem}\label{VarPar} Let $T>0$, $f\in L^{1}(0,T;H^{s})$, and $u:=\int_0^tW_{\mathbb{R}}(t-\tau)f(\tau)d\tau.$  Then, for any $s\in\mathbb{R}$, there exists a constant $C=C(s)>0$ such that \begin{equation}\label{Wb2Est1onRf}\sup_{t\in[0,T]}\|u(\cdot,t)\|_{H^s}+\sup_{x\in\mathbb{R}}\|u(x,\cdot)\|_{H^{\frac{2s+1}{4}}(0,T)}\le C \|f\|_{L^{1}(0,T;H^{s})}.\end{equation}

\end{lem}

\subsection{Regularity}  Combining Lemmas \ref{Wb1Lem}-\ref{VarPar}, we have the following regularity theorems for the linear model.
\begin{thm}\label{RegThm1}Let $T>0$, and $s\ge 1/2$.  Then, there exists $C>0$ (independent of T) such that for any $h\in H^{\frac{2s-1}{4}}(0,T)$ with $h(0)=0$ if $s>\frac{3}{2}$,  $u=W_b(t)h$ satisfies
\begin{multline}\sup_{t\in[0,T]}\|u(\cdot, t)\|_{H^s(\mathbb{R}_+)}+\sup_{x\in\mathbb{R}_+}\|u(x,\cdot)\|_{H^{\frac{2s+1}{4}}(0,T)}\\
\le C(1+T)\|h\|_{H^{\frac{2s-1}{4}}(0,T)}.\end{multline}
\end{thm}

\begin{thm}Let $T>0$, $s\in \left(\frac{1}{2},\frac{7}{2}\right)-\{3/2\}$, $h\in H^{\frac{2s-1}{4}}(0,T)$, $f\in L^1(0,T;H^s(\mathbb{R}_+))$, $u_0\in H^s(\mathbb{R}_+)$, and if $s\in\left(\frac{3}{2},\frac{7}{2}\right)$, we assume the zeroth order compatibility condition $u_0'(0)=h(0)$.  Then there exists $C>0$ (independent of $T$) such that the solution $u$ of \eqref{LinNonHom2} satisfies
\begin{multline}\sup_{t\in[0,T]}\|u(x,\cdot)\|_{H^s(\mathbb{R}_+)}+\sup_{x\in\mathbb{R}_+}\|u(x,\cdot)\|_{H^{\frac{2s+1}{4}}(0,T)}\\
\le C\left(\|u_0\|_{H^s(\mathbb{R}_+)}+(1+T)\|h\|_{H^{\frac{2s-1}{4}}(0,T)}+\|f\|_{L^{1}(0,T;H^s(\mathbb{R}_+))}\right).\end{multline}
\end{thm}

\begin{rem}The optimal local smoothing estimate for the Schrödinger evolution operator is $\|W_{\mathbb{R}}u_0\|_{L_x^\infty \dot{H}_t^\frac{2s+1}{4}}\lesssim \|u_0\|_{\dot{H}^s}$; see for instance \cite{KenigPonceVega1991}.  This is why we consider the space $X_T^s$ defined in Section \ref{Intro} as our solution space.  It is shown in \cite{Holmer2005} and \cite{BonaSunZhang2015} that the natural space for the boundary data $h$ is $H_t^{\frac{2s+1}{4}}(0,T)$, when one considers Dirichlet boundary conditions.  Since one can formally think that one derivative in the space variable is equivalent to $1/2$ derivatives in the time variable, we are inclined to consider $H_t^{\frac{2s-1}{4}}(0,T)$ as the natural space for the boundary data  $h$ when we consider Neumann boundary conditions.\end{rem}

\section{Nonlinear Schrödinger equation}\label{SectionNonLin}
In this section, we study nonlinear Schrödinger equations with nonhomogeneous Neumann type boundary data.  More precisely, we consider the following model:
\begin{equation}\label{MainProb2} \left\{ \begin{array}{ll}
         i\partial_t u - \partial_x^2u+f(u)= 0, & \mbox{$x\in \mathbb{R}_+$, $t\in (0,T)$},\\
         u(x,0)=u_0,\\
         \partial_xu(0,t)=h,\end{array} \right.
         \end{equation} where $f(u)=k|u|^pu$,  $p>0$, $k\in \mathbb{R}-\{0\}$, $u_0\in H^s(\mathbb{R}_+)$, and ${s\in \left(\frac{1}{2},\frac{7}{2}\right)-\{\frac{3}{2}\}}.$

         Here, we consider two problems.  The first one is the open-loop well-posedness problem when $h$ is taken as a time dependent function in the Sobolev space $H^{\frac{2s-1}{4}}(0,T)$.  The second one is the closed-loop well-posedness problem when $h$ is taken as a function of $u(0,t)$ in the form $h(u(0,t))=-\lambda|u(0,t)|^ru(0,t)$ with $\lambda\in \mathbb{R}-\{0\}$.

\subsection{Local existence}\label{LocExt} In order to prove the local existence of solutions we will use the contraction mapping argument.  For the contraction mapping argument, we will use the following operator on a closed ball $\bar{B}_R(0)$ in the function space $X_{T_0}^s$ for appropriately chosen $R>0$ and $T_0\in (0,T]$.
\begin{multline}\label{Psi}[\Psi(u)](t):=W_\mathbb{R}(t)u^*_0-i\int_0^tW_\mathbb{R}(t-\tau)f(u^*(\tau))d\tau\\+W_b(t)([h-g-p(u^*)]_e)\end{multline} with $g(t)=\partial_x W_\mathbb{R}(t)u^*_0|_{x=0}$ and $[p(u^*)](t)=-i\partial_x\int_0^tW_\mathbb{R}(t-\tau)f(u^*(\tau))d\tau|_{x=0}.$  Here, $g(t)$ and $p(t)$ make sense only if $s>3/2$.  For $s\in \left(\frac{1}{2},\frac{3}{2}\right)$, we take these boundary traces equal to zero in \eqref{Psi}.

In order to use the Banach fixed point theorem, we have to show that $\Psi$ maps $\bar{B}_R(0)$ onto itself, and moreover that it is a contraction on the same set.
Therefore, we will estimate each term in \eqref{Psi} with respect to the norm defined in \eqref{Xtsnorm}.  By Lemma \ref{RTheory1},
\begin{equation}\|W_\mathbb{R}(t)u^*_0\|_{X_T^s}\lesssim \|u_0^*\|_{H^s}\lesssim \|u_0\|_{H^s{(\mathbb{R}_+)}}.\end{equation}

In order to estimate the second term at the right hand side of \eqref{Psi}, we will first prove the following lemma:
\begin{lem}[Nonlinearity]\label{fEst} Let $f(u)=|u|^pu$ and $s>\frac{1}{2}$.  Moreover, let $(p,s)$ satisfy one of the following assumptions:
\begin{itemize}
  \item[(a1)] If $s$ is integer, then assume that $p\ge s$ if $p$ is an odd integer and $[p]\ge s-1$ if $p$ is non-integer.
  \item[(a2)] If $s$ is non-integer, then assume that $p>s$ if $p$ is an odd integer and $[p]\ge [s]$ if $p$ is non-integer.
\end{itemize}

If $u,v\in H^s$,  then \begin{equation}\label{fEstLem1}\|f(u)\|_{H^s}\lesssim\|u\|_{H^s}^{p+1},\end{equation}
\begin{equation}\|f(u)-f(v)\|_{H^s}\lesssim (\|u\|_{H^s}^{p}+\|v\|_{H^s}^{p})\|u-v\|_{H^s}.\end{equation}
\end{lem}

\begin{proof}  See Lemma 4.10.2 \cite{Cazenave2003} for $s$ being an integer and Lemma 3.10(2) \cite{Kri} for $p$ being an even number.    Therefore, we will only consider the cases with $s$ being a non-integer, and $p$ being an odd integer or non-integer.

Let us first consider the case $1/2<s<1$.  By the chain rule (Theorem A.7 \cite{KPV93b}), $\|D^sf(u)\|_{L^2}\lesssim \|f'(u)\|_{L^\infty}\|D^su\|_{L^2}.$ Since $|f'(u)|\lesssim |u|^{p}$, we have $\|f'(u)\|_{L^\infty}\lesssim \|u\|_{L^\infty}^p\lesssim \|u\|_{H^s}^p$ where the last inequality follows by the Sobolev embedding $H^s\hookrightarrow L^\infty$ for $s>1/2$.  Also, $\|D^su\|_{L^2}\le \|u\|_{H^s}.$  It follows that $\|D^sf(u)\|_{L^2}\lesssim \|u\|_{H^s}^{p+1}$. On the other hand, $\|f(u)\|_{L^2}=\|u\|_{L^{2p+2}}^{p+1}\lesssim \|u\|_{H^s}^{p+1}$, where the inequality follows by the Sobolev's embedding $H^s\hookrightarrow L^{2p+2}$ for $s>\frac{1}{2}$.  Hence, we have just shown that $\|f(u)\|_{H^s}\lesssim \|u\|_{H^s}^{p+1}$.

Now, consider the case $s=\sigma+m>1$ for some positive integer $m$ and $\sigma\in (0,1)$.  Then, $\|D^sf(u)\|_{L^2}\lesssim \|D^\sigma (D^mf(u))\|_{L^2}$ where $D^mf(u)$ is a sum of the terms of type $f^{(k)}(u)\prod_{j=1}^kD^{\beta_j}u$ where $k$ ranges from $k=1$ up to $k=m$ and $\sum_{j=1}^k\beta_j=m$.


By the fractional version of the Leibniz rule \cite{KPV93b}, we can write \begin{multline}\label{Impiden}\|D^\sigma (f^{(k)}(u)\prod_{j=1}^kD^{\beta_j}u)\|_{L^2}\lesssim\|D^\sigma(f^{(k)}(u))\|_{L^{p_1}}\|\prod_{j=1}^kD^{\beta_j}u\|_{L^{p_2}}\\
+\|f^{(k)}(u)\|_{L^\infty}\|D^\sigma(\prod_{j=1}^kD^{\beta_j}u)\|_{L^2}=I\cdot II+III\cdot IV. \end{multline} together with $\frac{1}{2}=\frac{1}{p_1}+\frac{1}{p_2}$, $p_1,p_2>2$.  By using the chain rule, the first term is estimated as $I\lesssim \|f^{(k+1)}(u)\|_{L^{q_1}}\|D^\sigma u\|_{L^{q_2}}$ together with $\frac{1}{p_1}=\frac{1}{q_1}+\frac{1}{q_2}$, $q_1,q_2>p_1>2$.  Here, we choose $q_1$ sufficiently large so that $q_1(p-k)>2$. Therefore, $\|f^{(k+1)}(u)\|_{L^{q_1}}\lesssim  \|u\|_{L^{q_1(p-k)}}^{p-k}\lesssim \|u\|_{H^s}^{p-k}$ and $\|D^\sigma u\|_{L^{q_2}}\le \|D^\sigma u\|_{H^m}\lesssim \|u\|_{H^s}$.  If $k=1$ (therefore $\beta_1=m$), then the second term can be estimated as $II=\|D^{m}u\|_{L^{p_2}}\lesssim \|D^{m}u\|_{H^\sigma}\lesssim \|u\|_{H^s}$.  In the last estimate, if $\sigma<1/2$, then we choose $p_2$ as $\frac{1}{p_2}=\frac{1}{2}-\sigma$, otherwise we can use any $p_2>2$.  If $k>1$, then using Hölder's inequality $$\|\prod_{j=1}^kD^{\beta_j}u\|_{L^{p_2}}\le \prod_{j=1}^k\|D^{\beta_j}u\|_{L^{q_j}}\lesssim \prod_{j=1}^k\|D^{\beta_j}u\|_{H^{1+\sigma}}\lesssim \|u\|_{H^s}^k$$ where $\frac{1}{p_2}=\sum_{j=1}^k\frac{1}{q_j}$ and $q_j>2$.  Hence, it follows that we always have $I\cdot II\lesssim \|u\|_{H^s}^{p+1}$.
The third term can be easily estimated as $III\lesssim \|u\|_{L^\infty}^{p-k+1}.$  Regarding the fourth term, the case $k=1$ is trivial.  So let us consider the case $k>1$.  In this case, applying the Leibniz formula, we have $\|D^\sigma(\prod_{j=1}^kD^{\beta_j}u)\|_{L^2}\lesssim \sum_{l=1}^k\|D^{\sigma+\beta_l}u\|_{L^{q_l}}\prod_{j=1,j\neq l}^k\|D^{\beta_j}u\|_{L^{q_j}}$ for some $\{q_j>2,j=1,...,k\}$ such that $\sum_{j=1}^k\frac{1}{q_j}=\frac{1}{2}$.  But the right hand side of the last inequality is dominated by $$\sum_{l=1}^k\|D^{\sigma+\beta_l}u\|_{H^{m-\beta_l}}\prod_{j=1,j\neq l}^k\|D^{\beta_j}u\|_{H^{s-\beta_j}}\lesssim \|u\|_{H^s}^k.$$ Hence, it follows that $III\cdot IV\lesssim \|u\|_{H^s}^{p+1}$.
By the above estimates, we deduce \eqref{fEstLem1}.


Regarding the differences, let us first consider the case $1/2<s<1$ again.  Then, by using the fractional chain rule and the fact that $H^s\hookrightarrow L^{\infty}$, we get \begin{multline}\|D^sf(u)-D^sf(v)\|_{L^2}\lesssim \|f'(u)-f'(v)\|_{L^\infty}\|D^s u-D^sv\|_{L^2}\\
\lesssim (\|u\|_{L^\infty}^p+\|v\|_{L^\infty}^p)\|u-v\|_{H^s}\lesssim (\|u\|_{H^s}^p+\|v\|_{H^s}^p)\|u-v\|_{H^s}.\end{multline}

Now, we consider the case $s=\sigma+m>1$ for some positive integer $m$ and $\sigma\in (0,1)$.  Then, $$\|D^sf(u)-D^sf(v)\|_{L^2}\lesssim \|D^\sigma (D^m(f(u)-f(v)))\|_{L^2}$$ where $D^m\left(f(u)-f(v)\right)$ is a sum of the terms of type \begin{multline}f^{(k)}(u)\prod_{j=1}^kD^{\beta_j}u-f^{(k)}(v)\prod_{j=1}^kD^{\beta_j}v\\
=\left(f^{(k)}(u)-f^k(v)\right)\prod_{j=1}^kD^{\beta_j}u-f^{(k)}(v)\prod_{j=1}^kD^{\beta_j}w_j\end{multline} where $k$ ranges from $k=1$ up to $k=m$, $\sum_{j=1}^k\beta_j=m$, and $w_j$'s are equal to $u$ or $v$, except one, which is equal to $u-v$.  Now the $L^2$-norm of the term $D^\sigma (f^{(k)}(v)\prod_{j=1}^kD^{\beta_j}w_j)$ can be estimated in a manner similar to \eqref{Impiden} using the fractional Leibniz rule, except we also use several applications of Young's inequality to separate the products involving $u$ and $v$.  What remains is to estimate the term
$D^\sigma\left[\left(f^{(k)}(u)-f^k(v)\right)\prod_{j=1}^kD^{\beta_j}u\right]$, which can also be done as in \eqref{Impiden} using the fractional Leibniz rule. In order to do this, we also use the observation $$\|f^k(u)-f^k(v)\|_{L^\infty}\lesssim \left(\|u\|_{H^s}^{p-k}+\|v\|_{H^s}^{p-k}\right)\|u-v\|_{H^s},$$ which easily follows from the fact that $$|f^k(u)-f^k(v)|\lesssim \left(|u|^{p-k}+|v|^{p-k}\right)|u-v|$$ and the Sobolev embedding $H^s\hookrightarrow L^{\infty}$ for $s>1/2$.

\end{proof}

\begin{rem}The assumption $(a1)$ and $(a2)$ are needed to guarantee that $f$ is sufficiently smooth.  The assumption $(a1)$ guarantees that $f$ is at least $C^m(\mathbb{C},\mathbb{C})$, which is what one needs in the case $s$ is an integer (see Remark 4.10.3 \cite{Cazenave2003}).  Since $f$ is $C^\infty(\mathbb{C},\mathbb{C})$ when $p$ is even, no assumption was necessary in this case. If $s$ is fractional, the proof uses the $m+1$-th derivative, which forces us to make the second assumption $(a2)$.  \end{rem}

It follows from Lemma \ref{VarPar} and Lemma \ref{fEst} that
\begin{multline}\left\|-i\int_0^tW_\mathbb{R}(t-\tau)f(u^*(\tau))d\tau\right\|_{X_T^s}\le \int_0^T\|f(u^*(\tau))\|_{H^s}d\tau\\
\lesssim \int_0^T\|u^*(\tau)\|_{H^s}^{p+1}d\tau\lesssim \int_0^T\|u(\tau)\|_{H^s(\mathbb{R}_+)}^{p+1}d\tau\le T\|u\|_{X_T^s}^{p+1}.\end{multline}

Similarly,
\begin{multline}\left\|-i\int_0^tW_\mathbb{R}(t-\tau)[f(u^*(\tau))-f(v^*(\tau))]d\tau\right\|_{X_T^s}\\
\le \int_0^T\|f(u^*(\tau))-f(v^*(\tau))\|_{H^s}d\tau\\
\lesssim \int_0^T(\|u^*(\tau)\|_{H^s}^p+\|v^*(\tau)\|_{H^s}^p)\|u^*(\tau)-v^*(\tau)\|_{H^s}d\tau\\
\lesssim \int_0^T(\|u(\tau)\|_{H^s(\mathbb{R}_+)}^p+\|v(\tau)\|_{H^s(\mathbb{R}_+)}^p)\|u(\tau)-v(\tau)\|_{H^s(\mathbb{R}_+)}d\tau\\
\lesssim T(\|u\|_{X_T^s}^p+\|v\|_{X_T^s}^p)\|u-v\|_{X_T^s}.\end{multline}

For $s\in \left(\frac{1}{2},\frac{3}{2}\right)$, since $g=p=0$, the last term in \eqref{Psi} is estimated as follows by using Theorem \ref{RegThm1}.
\begin{equation}\label{WbEst1}\|W_b(\cdot)h_e\|_{X_T^s}\le C(1+T)\|h\|_{H^{\frac{2s-1}{4}}(0,T)}.
\end{equation}

For $s\in \left(\frac{3}{2},\frac{7}{2}\right)$, we have the assumption $h(0)=u_0'(0)$, and therefore $h-g-p$ vanishes at $x=0$.  Moreover, the following estimate holds true.
\begin{multline}\label{Wbhgp}
\|W_b(\cdot)([h-g-p]_e)\|_{X_T^s}\le C(1+T)\|h-g-p(u^*)\|_{H^{\frac{2s-1}{4}}(0,T)}\\
\le C(1+T)\left(\|h\|_{H^{\frac{2s-1}{4}}(0,T)}+\|g\|_{H^{\frac{2s-1}{4}}(0,T)}+\|p(u^*)\|_{H^{\frac{2s-1}{4}}(0,T)}\right).
\end{multline}

Note that, \begin{multline}\label{gEst}\|g\|_{H^{\frac{2s-1}{4}}(0,T)}=\|\partial_x W_\mathbb{R}(t)u^*_0|_{x=0}\|_{H^{\frac{2s-1}{4}}(0,T)}\\
\le \sup_{x\in\mathbb{R}_+}\|\partial_x W_\mathbb{R}(t)u^*_0\|_{H^{\frac{2s-1}{4}}(0,T)}\\
\le \|\frac{d}{dx}u_0^*\|_{H^{s-1}}\le \|u_0^*\|_{H^{s}}\lesssim \|u_0\|_{H^s(\mathbb{R}_+)}.\end{multline}

In \eqref{gEst}, the second inequality follows from Lemma \ref{RTheory1} and the fact that $\partial_x W_\mathbb{R}(t)u^*_0$ is a solution of the linear Schrödinger equation on $\mathbb{R}$ with initial condition $\frac{d}{dx}u_0^*$.

Similarly,
\begin{multline}\label{pEst}\|p(u^*)\|_{H^{\frac{2s-1}{4}}(0,T)}=\|-i\partial_x\int_0^tW_\mathbb{R}(t-\tau)f(u^*(\tau))d\tau|_{x=0}\|_{H^{\frac{2s-1}{4}}(0,T)}\\
\le \sup_{x\in\mathbb{R}_+}\|-i\partial_x\int_0^tW_\mathbb{R}(t-\tau)f(u^*(\tau))d\tau\|_{H^{\frac{2s-1}{4}}(0,T)}\\
\le \|\partial_xf(u^*)\|_{L^1(0,T;H^{s-1})}\le \|f(u^*)\|_{L^1(0,T;H^{s})}\lesssim T\|u\|_{X_T^s}^{p+1}.\end{multline}

The last term in \eqref{Wbhgp}, \begin{multline}\|p(u^*)-p(v^*)\|_{H^{\frac{2s-1}{4}}(0,T)}\\
\lesssim T(\|u(\tau)\|_{X_T^s}^p+\|v(\tau)\|_{X_T^s}^p)\|u(\tau)-v(\tau)\|_{X_T^s}. \end{multline}

Combining above estimates, we obtain $$\|\Psi(u)\|_{X_T^s}\le C\left(\|u_0\|_{H^s(\mathbb{R}_+)}+(1+T)\|h\|_{H^{\frac{2s-1}{4}}(0,T)}+T\|u\|_{X_T^s}^{p+1}\right).$$

Similarly, regarding the differences, again by above estimates, we have $$\|\Psi(u)-\Psi(v)\|_{X_T^s}\le C\left( T(\|u(\tau)\|_{X_T^s}^p+\|v(\tau)\|_{X_T^s}^p)\|u(\tau)-v(\tau)\|_{X_T^s}\right).$$

Now, let $A:=C\left(\|u_0\|_{H^s(\mathbb{R}_+)}+(1+T)\|h\|_{H^{\frac{2s-1}{4}}(0,T)}\right)$, $R=2A$ and $T$ be small enough that $A+CTR^{p+1}<2A$.  Now, if necessary we can choose $T$ even smaller so that $\Psi$ becomes a contraction on $\bar{B}_R(0)\subset X_T^s$, which is a complete space.  Hence, $\Psi$ must have a unique fixed point in $\bar{B}_R(0)$ when we look for a solution whose lifespan is sufficiently small.

We conclude this section with the proposition below.

\begin{prop}\label{PropExt}Let $T>0$, $s\in \left(\frac{1}{2},\frac{7}{2}\right) -\left\{\frac{3}{2}\right\}$, $p,r>0$, $u_0\in H^s(\mathbb{R_+})$, $h\in H^{\frac{2s-1}{4}}(0,T)$, and $u_0'(0)=h(0)$ whenever $s>\frac{3}{2}$.  We in addition assume (a1)-(a2) given in Lemma \ref{fEst}.  Then, \eqref{MainProb2} has a local solution $u\in X_{T_0}^s$ for some ${T_0\in (0,T]}$.\end{prop}

\subsection{Uniqueness}\label{Uniq} In the previous section, we have proved uniqueness in a fixed ball in the space $X_T^s$.  This does not immediately tell us that the solution must also be unique in the entire space.  Fortunately, this latter statement is also true.  In order to show this, let $u_1,u_2\in X_{T_0}^s$ be two solutions of \eqref{MainProb2}.  Then, \begin{multline}u_1(t)-u_2(t)=-i\int_0^tW_\mathbb{R}(t-s)[f(u_1^*(s))-f(u_2^*(s))]ds\\
+W_b(t)([p(u_2^*)-p(u_1^*)]_e)\end{multline} for a.a. $t\in [0,T_0]$.

Since $s>1/2$,  \begin{multline}\|u_1(t)-u_2(t)\|_{H^s}\\
\le \int_0^{T_0}\|f(u_1^*(s))-f(u_2^*(s))\|_{H^s}+C(1+T_0)\|p(u_2^*)-p(u_1^*)\|_{H^{\frac{2s-1}{4}}(0,T)}\\
\le C(1+T_0)\int_0^{T_0}\|u_1(s)-u_2(s)\|_{H^s}(\|u_1(s)\|_{H^s}^p+\|u_2(s)\|_{H^s}^p)ds\\
\le C(1+T_0)(\|u_1(s)\|_{X_{T_0}^s}^p+\|u_2(s)\|_{X_{T_0}^s}^p)\int_0^{T_0}\|u_1(s)-u_2(s)\|_{H^s}ds.\end{multline}  By Gronwall's inequality, $\|u_1(t)-u_2(t)\|_{H^s}=0$, which implies $u_1\equiv u_2$.

Now, we can state the uniqueness statement as follows.
\begin{prop}If $u_1,u_2$ are two local solutions of \eqref{MainProb2} in $X_{T_0}^s$ as in Proposition \ref{PropExt}, then $u_1\equiv u_2$.\end{prop}

\subsection{Continuous dependence}\label{ContDep}
Regarding continuous dependence on data, let $B$ be a bounded subset of $H^s(\mathbb{R}_+)\times H^{\frac{2s-1}{4}}(0,T)$.
Let $(u_0, h_1)\in B$ and $(v_0,h_2)\in B$.  Let $u, v$ be two solutions on a common time interval $(0,T_0)$ corresponding to  $(u_0,h_1)$ and $(v_0,h_2)$, respectively.  Then $w=u-v$ satisfies
\begin{equation}\label{StabilityModel} \left\{ \begin{array}{ll}
         i\partial_t w + \partial_x^2w = F(x,t)\equiv f(v)-f(u), & \mbox{$x\in \mathbb{R}_+$, $t\in (0,T)$},\\
         w(x,0)=w_0(x)\equiv (u_0-v_0)(x),\\
         \partial_xw(0,t)=h(t)\equiv(h_1-h_2)(t).\end{array} \right.
         \end{equation}
Now, using the linear theory together with the nonlinear $H^s$ estimates on the differences, we have
$$\|w\|_{X_{T_0}^s}\le C\left(\|w_0\|_{H^s(\mathbb{R}_+)}+(1+T_0)\|h\|_{H^{\frac{2s-1}{4}}(0,T)}+\|F\|_{L^1(0,T_0;H^s(\mathbb{R}_+))}\right),$$
where $$\|F\|_{L^1(0,T_0;H^s(\mathbb{R}_+))}\le CT_0\left(\|u\|_{X_{T_0}^s}^p+\|v\|_{X_{T_0}^s}^p\right)\|u-v\|_{X_{T_0}^s}.$$

Choosing $R$, which depends on $u_0$ and $h$ (i.e., on the bounded set $B$), as in the proof of the local existence, and $T_0$ accordingly small enough, we obtain
\begin{equation}\label{ContDep01}\|u-v\|_{X_{T_0}^s}\le C\left(\|u_0-v_0\|_{H^s(\mathbb{R}_+)}+\|h_1-h_2\|_{H^{\frac{2s-1}{4}}(0,T)}\right).\end{equation}
Hence, we have the following result.

\begin{prop} If $B$ is a bounded subset of $H^s(\mathbb{R}_+)\times H^{\frac{2s-1}{4}}(0,T)$, then there is $T_0>0$ such that the flow $(u_0,h)\rightarrow u$ is Lipschitz continuous from $B$ into $X_{T_0}^s$.\end{prop}
\subsection{Blow-up alternative}\label{BlowSec} In this section, we want to obtain a condition which guarantees that a given local solution on $[0,T_0]$ can be extended globally.  Let's consider the set $S$ of all $T_0\in (0,T]$ such that there exists a unique local solution in $X_{T_0}^s$.  We claim that if $\displaystyle T_{max}:=\sup_{T_0\in S}T_0<T$, then $\displaystyle\lim_{t\uparrow T_{max}}\|u(t)\|_{H^s(\mathbb{R}_+)}=\infty.$  In order to prove the claim, assume to the contrary that $\displaystyle\lim_{t\uparrow T_{max}}\|u(t)\|_{H^s(\mathbb{R}_+)}\neq\infty.$  Then $\exists M$ and $t_n\in S$ such that $t_n\rightarrow T_{max}$ and $\|u(t_n)\|_{H^s(\mathbb{R}_+)}\le M.$  For a fixed $n$, we know that there is a unique local solution $u_1$ on $[0,t_n]$.  Now, we consider the following model.

\begin{equation}\label{MainProb3} \left\{ \begin{array}{ll}
         i\partial_t u - \partial_x^2u+f(u)= 0, & \mbox{$x\in \mathbb{R}_+$, $t\in (t_n,T)$},\\
         u(x,t_n)=u_1(x,t_n),\\
         \partial_xu(0,t)=h(t).\end{array} \right.
         \end{equation}
We know from the local existence theory that the above model has a unique local solution $u_2$ on some interval $[t_n,t_n+\delta]$ for some $\delta=\delta\left(M,\|h\|_{H^\frac{2s-1}{4}(0,T)}\right)\in (0,T-t_n].$  Now, choose $n$ sufficiently large that $t_n+\delta>T_{max}$.  If we set \begin{equation} u:=\left\{ \begin{array}{ll}
        u_1, & \mbox{$t\in [0,t_n)$},\\
         u_2, & \mbox{$t\in [t_n,t_n+\delta]$},\end{array} \right.\end{equation} then $u$ is a solution of \eqref{MainProb2} on $[0,t_n+\delta]$ where $t_n+\delta>T_{max}$, which is a contradiction.

We have the theorem.

\begin{prop}\label{BlowupAlt1}Let $S$ be the set of all $T_0\in (0,T]$ such that there exists a unique local solution in $X_{T_0}^s$.  If $\displaystyle T_{max}:=\sup_{T_0\in S}T_0<T$, then $\displaystyle\lim_{t\uparrow T_{max}}\|u(t)\|_{H^s(\mathbb{R}_+)}=\infty.$\end{prop}

\subsection{Nonlinear boundary data}\label{SectionNonlin2}
In this section, we study the most general nonlinear model given in \eqref{MainProb}.  We define the operator $\Psi$ as in \eqref{Psi}, except that we take $h(t)=h(u(0,t))=-\lambda|u(0,t)|^ru(0,t).$   Therefore, the solution operator we have to use for the contraction argument takes the following form.

\begin{multline}\label{PsiNonlinear}[\Psi(u)](t):=W_\mathbb{R}(t)u^*_0-i\int_0^tW_\mathbb{R}(t-\tau)f(u^*(\tau))d\tau\\+W_b(t)([h(u(0,\cdot))-g-p(u^*)]_e)\end{multline}

The proofs of local well-posedness and blow-up alternative now follows similar to the proofs in Sections \ref{LocExt} - \ref{BlowSec}.  The only additional work in this part would be to get nonlinear $H^s$ estimates on the boundary trace $-\lambda|u(0,t)|u(0,t)$, which is of course possible with assumptions on $r$, which are almost equivalent to the assumptions we made on $p$.  Indeed,  we will assume that $r>\frac{2s-1}{4}$ if $r$ is an odd integer and $[r]\ge \left[\frac{2s-1}{4}\right]$ if $r$ is non-integer.

We will need the following lemma to get useful estimates on the boundary operator for the contraction argument.
\begin{lem}\label{fTEst}Let $h\in H^{\sigma+\epsilon}(0,T)$, $\sigma,\epsilon>0$. Then  $\|h\|_{H^\sigma(0,T)}\le T^{\frac{\epsilon}{1+\sigma+\epsilon}}\|h\|_{H^{\sigma+\epsilon}(0,T)}$.\end{lem}
\begin{proof}
Let $H(t):=\int_0^t h(s)ds$ . Applying the Cauchy-Schwartz inequality we get
$\|H\|_{L^2(0,T)}^2\leq\int_0^T\left(\int_0^T|h(s)|ds\right)^2 dt\leq T^2\|h\|_{L^2(0,T)}^2$. On the other hand $H'=h$, which implies $\|h\|_{H^{-1}(0,T)}\leq \|H\|_{L^2(0,T)}$, hence $\|h\|_{H^{-1}(0,T)}\leq T\|h\|_{L^2(0,T)}$. By interpolation theorem \cite[Theorem 12.4, Proposition 2.3]{LionMag}, $\|h\|_{H^\sigma}\le \|h\|_{H^{-1}}^{\theta}\cdot\|h\|_{H^{\sigma+\epsilon}}^{1-\theta}$, in which $\theta=\frac{\epsilon}{1+\sigma+\epsilon}$.
Hence we obtain $\|h\|_{H^\sigma}\le T^\theta\|h\|^\theta_{L^2(0,T)}\cdot\|h\|_{H^{\sigma+\epsilon}}^{1-\theta}\leq T^\theta\|h\|_{H^{\sigma+\epsilon}}.$
\end{proof}

Let us first consider the case $r$ being an odd integer.  In this case, we assume $r>\frac{2s-1}{4}$.  Now, if $\frac{2s-1}{4}<\frac{1}{2}$, then we can choose $\epsilon=\frac{1}{2}$ so that $\frac{1}{2}<\frac{2s-1}{4}+\epsilon< 1\le r$.  If $\frac{2s-1}{4}>\frac{1}{2}$, then we can choose $\epsilon$ sufficiently small so that we again have  $\frac{1}{2}<\frac{2s-1}{4}+\epsilon\le r$.

Secondly, let us consider the situation for $r>0$ being a non-integer.  In this case, we assume $[r]\ge \left[\frac{2s-1}{4}\right]$.  If $\frac{2s-1}{4}<\frac{1}{2}$ then we choose $\epsilon=\frac{1}{2}$ so that $\frac{1}{2}<\frac{2s-1}{4}+\epsilon$.  If $\frac{2s-1}{4}>\frac{1}{2}$, then we choose $\epsilon$ sufficiently small so that $\left[\frac{2s-1}{4}+\epsilon\right]=\left[\frac{2s-1}{4}\right].$

If $r$ is even and $\frac{2s-1}{4}<\frac{1}{2}$, then again we choose $\epsilon=\frac{1}{2}$.

Now, given $u\in X_T^s$, we know that $u(0,\cdot)$ in particular belongs to the space $H^{\frac{2s-1}{4}+\epsilon}(0,T)$ for any $\epsilon\in(0,\frac{1}{2}]$.  So, let us take an extension of $u(0,\cdot)\in H^{\frac{2s-1}{4}+\epsilon}(0,T)$, say  $U\in H^{\frac{2s-1}{4}+\epsilon}$, so that \begin{equation}\label{Uvsu}\|U\|_{H^{\frac{2s-1}{4}+\epsilon}}\le 2\|u(0,\cdot)\|_{H^{\frac{2s-1}{4}+\epsilon}(0,T)},\end{equation}  see \eqref{infDef}.  Now, $|U|^rU$ is an extension of $|u(0,\cdot)|^ru(0,\cdot)$, and therefore $$\||u(0,\cdot)|^ru(0,\cdot)\|_{H^{\frac{2s-1}{4}+\epsilon}(0,T)}\le \||U|^rU\|_{H^{\frac{2s-1}{4}+\epsilon}}\lesssim \|U\|_{H^{\frac{2s-1}{4}+\epsilon}}^{r+1}$$ by Lemma \ref{fEst}.  By using the inequality \eqref{Uvsu}, we have $$\|U\|_{H^{\frac{2s-1}{4}+\epsilon}}^{r+1}\lesssim \|u(0,\cdot)\|_{H^{\frac{2s-1}{4}+\epsilon}(0,T)}^{r+1}.$$ Combining the above estimates, we arrive at
$$\||u(0,\cdot)|^ru(0,\cdot)\|_{H^{\frac{2s-1}{4}+\epsilon}(0,T)}\lesssim \|u(0,\cdot)\|_{H^{\frac{2s-1}{4}+\epsilon}(0,T)}^{r+1}\le \|u(0,\cdot)\|_{H^{\frac{2s+1}{4}}(0,T)}^{r+1}.$$

Finally, we deduce that
\begin{multline}\label{hNonlinearEst}\|W_b(\cdot)[h(u(0,\cdot)]_e\|_{X_T^s}\le C(1+T)\||u(0,\cdot)|^ru(0,\cdot)\|_{H^{\frac{2s-1}{4}}(0,T)}\\
\le C(1+T)T^{\frac{4\epsilon}{2s+3+4\epsilon}}\||u(0,\cdot)|^ru(0,\cdot)\|_{H^{\frac{2s-1}{4}+\epsilon}(0,T)}\\
\le C(1+T)T^{\frac{4\epsilon}{2s+3+4\epsilon}}\|u(0,\cdot)\|_{H^{\frac{2s+1}{4}}(0,T)}^{r+1}\\
\le C(1+T)T^{\frac{4\epsilon}{2s+3+4\epsilon}}\sup_{x\in\mathbb{R}_+}\|u(x,\cdot)\|_{H^{\frac{2s+1}{4}}(0,T)}^{r+1}\\
\le C(1+T)T^{\frac{4\epsilon}{2s+3+4\epsilon}}\|u\|_{X_T^s}^{r+1}.\end{multline}

We can estimate the differences similarly.  Namely, for any given $u,v\in X_T^s$, we have
\begin{multline}\label{hNonlinearEst2}\|W_b(\cdot)[h(u(0,\cdot)-h(v(0,\cdot)]_e\|_{X_T^s}\\
\le C(1+T)T^{\frac{4\epsilon}{2s+3+4\epsilon}}\left(\|u\|_{X_T^s}^{r}+\|v\|_{X_T^s}^{r}\right)\|u-v\|_{X_T^s}.\end{multline}

\subsubsection*{Local Existence}

Following the arguments in Section \ref{LocExt} and using the estimate \eqref{hNonlinearEst}, we have $$\|\Psi(u)\|_{X_T^s}\le C\left(\|u_0\|_{H^s(\mathbb{R}_+)}+(1+T)T^{\frac{4\epsilon}{2s+3+4\epsilon}}\|u\|_{X_T^s}^{r+1}+T\|u\|_{X_T^s}^{p+1}\right).$$

On the other hand, using the estimate \eqref{hNonlinearEst2} and the arguments in Section \ref{LocExt} for differences, we obtain

\begin{multline}\|\Psi(u)-\Psi(v)\|_{X_T^s}\le C\left(T(\|u(\tau)\|_{X_T^s}^p+\|v(\tau)\|_{X_T^s}^p)\|u(\tau)-v(\tau)\|_{X_T^s}\right.\\
\left.+(1+T)T^{\frac{4\epsilon}{2s+3+4\epsilon}}\left(\|u\|_{X_T^s}^{r}+\|v\|_{X_T^s}^{r}\right)\|u-v\|_{X_T^s}\right).\end{multline}

Now, let $A:=C\left(\|u_0\|_{H^s(\mathbb{R}_+)}\right)$, $R=2A$ and $T$ small enough that $$A+C(1+T)T^{\frac{4\epsilon}{2s+3+4\epsilon}}R^{r+1}+CTR^{p+1}<2A.$$  Now, if necessary we can choose $T$ even smaller so that $\Psi$ becomes a contraction on $\bar{B}_R(0)\subset X_T^s$, which is a complete space.  Hence, $\Psi$ must have a unique fixed point in $\bar{B}_R(0)$ when we look for a solution whose lifespan is sufficiently small.
\subsubsection*{Uniqueness}  In order to prove uniqueness, we proceed as in Section \ref{Uniq}, taking into account that the boundary forcing now depends on the solution itself.  So, let $u_1,u_2\in X_{T_0}^s$ be two solutions of \eqref{MainProb}.  Then, \begin{multline}u_1(t)-u_2(t)=-i\int_0^tW_\mathbb{R}(t-s)[f(u_1^*(s))-f(u_2^*(s))]ds\\
+W_b(t)\left([h(u_1(0,\cdot))-h(u_2(0,\cdot))+p(u_2^*)-p(u_1^*)]_e\right)\end{multline} for a.a. $t\in [0,T_0]$.  Then, \begin{multline}\|u_1(t)-u_2(t)\|_{H^s}\\
\le \int_0^{T_0}\|f(u_1^*(s))-f(u_2^*(s))\|_{H^s}+C(1+T_0)\|p(u_2^*)-p(u_1^*)\|_{H^{\frac{2s-1}{4}}(0,T)}\\
+C(1+T_0)T_0^{\frac{4\epsilon}{2s+3+4\epsilon}}\left(\|u_1\|_{X_{T_0}^s}^{r}+\|u_2\|_{X_{T_0}^s}^{r}\right)\|u_1-u_2\|_{H^s}.\end{multline}  Now, choosing $T_0$ sufficiently small, we can subtract the last term above from the left hand side, estimate the rest of terms at the right hand side as in Section \ref{Uniq}, and then use the Gronwall's inequality to obtain $\|u_1(t)-u_2(t)\|_{H^s}=0$.

\subsubsection*{Continuous Dependence}  The proof of continuous dependence can be done as in Section \ref{ContDep} by taking into account that $h$ is now a function of $u(0,t)$.  For this closed loop problem, the estimate \eqref{ContDep01} takes the following form.
\begin{equation}\label{ContDep02}\|u-v\|_{X_{T_0}^s}\le C\|u_0-v_0\|_{H^s(\mathbb{R}_+)}\end{equation} for sufficiently small $T_0$.  Of course, for the closed loop problem $B$ is taken as a subset of $H^s(\mathbb{R}_+)$ with finite diameter.
\subsubsection*{Blow-up Alternative}  The proof of the blow-up alternative is almost identical to the proof given in Section \ref{BlowSec}, and is therefore omitted here.  The only modification is that the parameter $\delta$ in the proof given in Section \ref{BlowSec} now depends only on $M$.

\end{document}